\documentclass[11pt,oneside]{amsart}
\usepackage{amsmath}
\usepackage{amssymb}
\usepackage{amscd}
\usepackage{amsthm}
\usepackage{amsxtra}
\usepackage{latexsym}
\usepackage{graphicx}
\usepackage{comment}
\numberwithin{equation}{section}
\theoremstyle{plain}
\newtheorem{theorem}[equation]{Theorem}

\newtheorem{proposition}[equation]{Proposition}

\newtheorem{lemma}[equation]{Lemma}

\newtheorem{corollary}[equation]{Corollary}

\theoremstyle{remark}

\theoremstyle{definition}
\newtheorem{definition}[equation]{Definition}

\newtheorem{question}[equation]{Question}

\newcommand{\C}{\mathbb C}

\newcommand{\K}{{\mathcal K}}

\def\Pic{\text{Pic}}
\renewcommand{\P}{\mathbb P}

\newcommand{\R}{\mathbb R}

\newcommand{\Z}{\mathbb Z}
\newcommand{\mcal}[1]{\mathcal{#1}}

\DeclareMathOperator{\Bs}{Bs}

\begin{document}

\title{Big q-Ample Line Bundles}
\author{Morgan V Brown}
\email{mvbrown@math.berkeley.edu}

\begin{abstract}
A recent paper of Totaro develops a theory of $q$-ample bundles in characteristic $0$.  Specifically, a line bundle $L$ on $X$ is $q$-ample if for every coherent sheaf $\mathcal{F}$ on $X$, there exists an integer $m_0$ such that $m\geq m_0$ implies $H^i(X,\mathcal{F}\otimes \mathcal{O}(mL))=0$ for $i>q$.  We show that a line bundle $L$ on a complex projective scheme $X$ is $q$-ample if and only if the restriction of $L$ to its augmented base locus is $q$-ample.  In particular, when $X$ is a variety and $L$ is big but fails to be $q$-ample, then there exists a codimension $1$ subscheme $D$ of $X$ such that the restriction of $L$ to $D$ is not $q$-ample.
\end{abstract}
\maketitle
\section{Introduction}

A recent paper of Totaro \cite{1007.3955} generalizes the notion of an ample line bundle, with the object of relating cohomological, numerical, and geometric properties of these line bundles.  
Let $q$ be a natural number.  Totaro calls a line bundle $L$ on $X$ $q$-ample if for every coherent sheaf $\mathcal{F}$ on $X$, there exists an integer $m_0$ such that $m\geq m_0$ implies $H^i(X,\mathcal{F}\otimes \mcal{O}(mL))=0$ for $i>q$.

Totaro \cite{1007.3955} has shown that in characteristic $0$, this notion of $q$-amplitude is equivalent to others previously studied by Demailly, Peternell, and Schneider in \cite{MR1360502}. As a result, the $q$-amplitude of a line bundle depends only on its numerical class, and the cone of such bundles is open.  This means that there is some hope of recovering geometric and numerical information about $X$ and its subvarieties from knowing when a line bundle is $q$-ample, though at present such results are known only in limited cases.  
In general much is known about the $0$-ample cone (which is the ample cone) and the $(n-1)$-ample cone of an $n$ dimensional variety $X$ is known to be the negative of the complement of the pseudoeffective cone of $X$.  For values of $q$ between $1$ and $n-2$ the relation between numerical and cohomological data remains mysterious.  The Kleiman criterion tells us that $0$-amplitude is determined by the restriction of $L$ to the irreducible curves on $X$, and likewise one gets at least some information about the $q$-ample cone by looking at restrictions to $(q+1)$-dimensional subvarieties.

However, Totaro \cite{1007.3955} has given an example of a smooth toric $3$-fold with a line bundle $L$ which is not in the closure of the $1$-ample cone, but the restriction of $L$ to every $2$-dimensional subvariety is in the closure of the $1$-ample cone of each subvariety.  For completeness, we include this example in section \ref{ex}.  The example shows that the most direct generalization of Kleiman's criterion does not hold for even the first open case:  the $1$-ample cone of a $3$-fold.  

The goal of this note is to show that one can in fact test $q$-amplitude on proper subschemes in the case where $L$ is a big line bundle on a projective variety $X$.  In particular, we show that if $L$ is a big line bundle which is not $q$-ample, and $D$ is the locus of vanishing of a negative twist of $L$, then the restriction of $L$ to $D$ is not $q$-ample either.  In a recent paper \cite{1012.1102}, K{\"u}ronya proves a sort of Fujita vanishing theorem for line bundles whose augmented base locus has dimension at most $q$.  As a consequence he shows that if the augmented base locus of $L$ has dimension $q$, then $L$ is $q$-ample.  We prove the following result:

\begin{theorem}\label{base}
Let $X$ be complex projective scheme, and let $L$ be a line bundle on $X$.  Let $Y$ be the scheme given by the augmented base locus of $L$ with the unique scheme structure as a reduced closed subscheme of $X$.  Then $L$ is $q$-ample on $X$ if and only if the restriction of $L$ to $Y$ is $q$-ample.
\end{theorem}

S. Matsumura has shown in \cite{1104.5313} that a line bundle admits a hermitian metric whose curvature form has all but $q$ eigenvalues positive at every point iff it admits such a metric when restricted to the augmented base locus.  A line bundle with such a metric is $q$-ample, but it is unknown in general whether every $q$-ample line bundle admits such a metric.

We also prove a Kleiman-type criterion for $(n-2)$-amplitude for big divisors when $X$ is smooth.

\begin{corollary}\label{criterion}
Let $X$ be a nonsingular projective variety.  A big line bundle $L$ on $X$ is $(n-2)$-ample iff the restriction of $-L$ to every irreducible codimension $1$ subvariety is not pseudoeffective.    
\end{corollary}

When $X$ is a $3$-fold, a big line bundle $L$ is $1$-ample iff its dual is not in the pseudoeffective cone when restricted to any surface contained in $X$.  Since a big line bundle on a $3$-fold is always $2$-ample, our results give a complete description of the intersection of the $q$-ample cones with the big cone of a $3$-fold in terms of restriction to subvarieties. 

In the final section we examine possible geometric criteria for an effective line bundle to be $q$-ample. In particular, on an $n$-dimensional Cohen Macaulay variety, any line bundle which admits a disconnected section must fail to be $(n-2)$-ample.  This fact in particular helps to explain some features of Totaro's example, and may lead to more general criteria for $q$-amplitude.

I would like to thank my advisor David Eisenbud as well as Alex K{\"u}ronya, Rob Lazarsfeld, and Burt Totaro for helpful discussions and comments. 
\section{The Restriction Theorem}

In this section we prove that a line bundle $L$ which fails to be $q$-ample is still not $q$-ample when restricted to any section of $L-H$, where $H$ is any ample line bundle.   
\begin{theorem}\label{restrict}
Let $X$ be a reduced projective scheme over $\C$.  Suppose $L$ is a line bundle on $X$ which is not $q$-ample on $X$, and let $L'$ be a line bundle with a nonzero section such that $\mcal{O}(\alpha L-\beta L')$ is ample for some positive integers $\alpha, \beta$.  Let $D$ be the subscheme of $X$ given by the vanishing of some nonzero section of $L'$.  Then $L|_{D}$ is not $q$-ample on $D$.
\end{theorem}

Before proving Theorem \ref{restrict}, we will need a lemma:
\begin{lemma}\label{mainlemma}
Let $X$ be a projective scheme over $\C$.  Fix an ample line bundle $H$ on $X$.  Suppose $L$ is a $q$-ample line bundle on $X$ for some $q\geq 0$.  Then for every coherent sheaf $\mathcal{F}$ on $X$ there exist integers $a_0$ and $b_0$ such that given $a,b\geq 0$, $H^i(X,\mathcal{F}\otimes\mcal{O}(aL+bH))=0$ for $i>q$ whenever $a\geq a_0$ or $b\geq b_0$.
\end{lemma}

\begin{proof}
Every coherent sheaf has a possibly infinite resolution by bundles of the form $\bigoplus\mathcal{O}(-dH)$.  By \cite[Appendix B]{MR2095471}, it thus suffices to check for finitely many sheaves of the form $\mcal{O}(-dH)$.
The proof follows by induction on the dimension of $X$.  In the base case, dimension $0$, the lemma follows because for every coherent sheaf the groups $H^i$ vanish for $i>0$.

Since every ample line bundle has some multiple which is very ample it suffices to prove the lemma when $H$ is very ample.  It is also enough to find the constants $a_0$ and $b_0$ such that the cohomology vanishes for a fixed $i>q$.  Assume $H$ is very ample, and fix an $i>q$.  Now, suppose $X$ has dimension $n$ and the lemma is true for projective schemes of dimension $n-1$.

Because $L$ is $q$-ample, we know there exists $a_1$ such that $H^i(X,\mcal{O}(aL-dH))=0$ whenever $a\geq a_1$.  Let $D$ be a hyperplane section under the embedding given by $H$.  By the inductive hypothesis, there exists $a_2$ such that $H^i(D,\mcal{O}(aL+(b-d)H))=0$ whenever $a\geq a_2$ and $b\geq 0$.  By abuse of notation, we use $L$ to refer to both the line bundle on $X$ and its pullback to $D$.  The projection formula \cite[II, Ex 5.1]{MR0463157} along with the preservation of cohomology under push forward by a closed immersion shows that this will not change the cohomology. 
Thus we have an exact sequence in cohomology:
\[
\ldots \to H^i(X,\mcal{O}(aL+(b-d)H)) \to H^i(X,\mcal{O}(aL+(b+1-d)H))\to H^i(D,\mcal{O}(aL+(b+1-d)H)_{|D}) \to \ldots
\]
Set $a_0=\text{max}\{a_1,a_2\}$.  Then for $a\geq a_0$, we know that $H^i(D,\mcal{O}((aL+(b+1-d)H))_{|D})=0$ so by induction on $b$ we know that $H^i(X,\mcal{O}(aL+(b-d)H))$ vanishes for all $b>0$.  To find $b_0$, we know that for each $a<a_0$, there exists  $b'$ such that the cohomology vanishes for $b>b'$ since $H$ is ample.  Take $b_0$ as the maximum of all the $b'$.  \end{proof}

\begin{proof}[Proof of Theorem \ref{restrict}]

$L$ is $q$-ample iff $\alpha L$ is, so we may assume $\alpha =1$.  Likewise in Totaro \cite[Cor 7.2]{1007.3955} Totaro shows that $L$ is $q$-ample on a scheme $X$ iff its restriction to the reduced scheme is $q$-ample, so we may assume $\beta=1$.

We recall another result of Totaro \cite[Thm 7.1]{1007.3955}:  
Given a fixed ample line bundle $H$ there exists a global constant $C$ such that $L$ is $q$-ample iff there exists $N$ such that $H^i(X,\mcal{O}(NL-jH))=0$ for all $i>q$, $1\leq j\leq C$.  Let us assume $L$ is $(q+1)$-ample but not $q$-ample.  Since $L$ is not $q$-ample for all $N$ one of the above groups is nonzero.  Since $L$ is $(q+1)$-ample that group must have $i=q+1$ for large enough $N$.  Now, $H$ is ample so for sufficiently large $e$, $H^i(X, \mcal{O}((e-j)H))=0$ for $i>q$, $1\leq j\leq C$.    

Likewise, for all sufficiently large $e\geq 1$, we know that $H^{q+1}(X,\mcal{O}((e-j)H))=0$, and that for some $1\leq j\leq C$, $H^{q+1}(X, \mcal{O}(eL-jH))\neq 0$.  Since $\mcal{O}(L')=\mcal{O}(L-H)$ there exist $j$ and $k$ such that $1\leq j \leq C$, and $1\leq k\leq e$ such that $H^{q+1}(X, \mcal{O}((e-j)H+(k-1)L'))=0$ and $H^{q+1}(X, \mcal{O}((e-j)H+kL'))\neq 0$.  To simplify notation we set $l=e-j$.

Consider the exact sequence:
\[
0\to \mcal{F} \to \mcal{O}_X(-L')\to \mcal{O}_X\to \mcal{O}_D\to 0
\]
The section defining $D$ may be given by a section which is not regular when $X$ is reducible and so the sheaf $\mcal{F}$ may be nonzero.    
Now write $\mcal{G}=\text{coker}(\mcal{F} \to \mcal{O}_X(-L'))=\text{ker}(\mcal{O}_X\to \mcal{O}_D)$.
After twisting by $\mcal{O}(lH+kL')$ we have two resulting long exact sequences in cohomology.  The first is
\[
\ldots \to H^{q+1}(X,\mcal{O}(lH+(k-1)L')) \to H^{q+1}(X, \mcal{G}\otimes \mcal{O}(lH+kL'))\to H^{q+2}(X, \mcal{F}\otimes\mcal{O}(lH+kL')) \ldots
\] 
Since $k\leq l$ and $\mcal{O}(H+L')=\mcal{O}(L)$, for sufficiently large $e$, $H^{q+2}(X, \mcal{F}\otimes\mcal{O}(lH+kL'))=H^{q+2}(X, \mcal{F}\otimes\mcal{O}((l-k)H+kL))=0$, by Lemma \ref{mainlemma} .  Thus $H^{q+1}(X,\mcal{O}(lH+(k-1)L'))=0$ implies  $H^{q+1}(X, \mcal{G}\otimes \mcal{O}(lH+kL'))=0$.

The second long exact sequence is given by
\[
\ldots \to H^i(X, \mcal{G}\otimes\mcal{O}(lH+kL')) \to H^i(X,\mcal{O}(lH+kL')) \to H^i(D, \mcal{O}(lH+kL')_{|D})\to \ldots
\]
The group $H^{q+1}(X, \mcal{G}\otimes \mcal{O}(lH+kL'))=0$, and $ H^i(X,\mcal{O}(lH+kL'))\neq 0$, so we see that $H^i(D, \mcal{O}(lH+kL')_{|D})\neq 0$.  $\mcal{O}(lH+kD)=\mcal{O}((l-k)H+kL)$, which has the form $\mcal{O}(aL+(b-d)H)$, where $d=C$, $a,b\geq 0$, and $a+b\geq e$.  Since we could choose $e$ arbitrarily large, by Lemma \ref{mainlemma} $L$ is not $q$-ample when restricted to $D$.

\end{proof}

In the case where $X$ is irreducible, every nonzero section of a line bundle is regular, and we get the following corollary:

\begin{corollary}\label{big}
If $X$ is a complex projective variety (irreducible and reduced) and $L$ is a big line bundle which is not $q$-ample, there exists a codimension $1$ subscheme of $X$ on which $L$ is not $q$-ample.
\end{corollary}

\begin{proof}
The cone of big line bundles on a projective variety is open, so we may pick $L'$ also big, so some large multiple of $L'$ has a nonzero section whose vanishing is an effective Cartier divisor.
\end{proof}
One subtlety of the Kleiman criterion for ample divisors is that is possible to have a divisor class which is positive on every irreducible curve but is not ample.  One such example is due to Mumford and can be found in \cite[Example 1.5.2]{MR2095471}.
In particular this shows that in Corollary \ref{big} the hypothesis `big' cannot be replaced by `pseudoeffective'.

\section{Augmented Base Loci}
Let $L$ be a Cartier divisor on a variety $X$.  Write $\Bs(|L|)$ for the base locus of the full linear series of $L$.  It is also helpful to have a notion of the base locus for large multiples of $L$, as well as for small perturbations by the inverse of an ample line bundle.
\begin{definition}\cite[Def 2.1.20]{MR2095471}
The stable base locus of $L$ is the algebraic set \[{\bf B}(L)=\bigcap_{m\geq 1} \Bs(|mL|).\] 
\end{definition}
There exists an integer $m_0$ such that ${\bf B}(L) = \Bs(|km_0L|)$ for $k>>0$ \cite[Prop 2.1.20]{MR2095471}.
\begin{definition}\cite[Def 10.3.2]{MR2095472}
The augmented base locus of $L$, denoted by $\mathbf{B}_{+}(L)$, is the closed algebraic set given by $\mathbf{B}(L-\epsilon \mathcal{H})$, for any ample $\mathcal{H}$, and sufficiently small $\epsilon>0$.
\end{definition}
It is a theorem of Nakamaye \cite{MR1802513} that the augmented base locus is well defined.  Note that stable and augmented base loci are defined as algebraic sets, not as schemes.

Geometric properties of $\mathbf{B}_{+}(L)$ reveal information about how much $L$ fails to be ample.  For example, $\mathbf{B}_{+}(L)$ is empty if and only if $L$ is ample.  More generally, K{\"u}ronya has proved in \cite{1012.1102} a Fujita-vanishing type result for the cohomology groups $H^i$ where $i>\text{dim} \mathbf{B}_{+}(L)$. 
\begin{theorem}\label{Kuronya}\cite[Thm C]{1012.1102}
Let $X$ be a projective scheme, $L$ a Cartier divisor, and $\mathcal{F}$ a coherent sheaf on $X$.
Then there exists $m_0$ such that $m\geq m_0$ implies $H^i(X,\mathcal{F}\otimes\mcal{O}(mL+D))=0$ for all $i>$dim $\mathbf{B}_{+}(L)$ and any nef divisor $D$.
\end{theorem}
In particular, K{\"u}ronya's theorem implies that $L$ is $q$-ample, for all $q$ at least as big as the dimension of $\mathbf{B}_{+}(L)$.  We show that in fact $L$ is $q$-ample if and only if the restriction of $L$ to $\mathbf{B}_{+}(L)$ is $q$-ample:

\begin{proof}[Proof of Theorem \ref{base}]
Certainly if $L$ is $q$-ample on $X$ it must be $q$-ample on $Y$.  For the converse, we apply \ref{restrict} inductively.  Suppose $L$ is not $q$-ample.  We may assume all schemes are reduced by \cite[Cor 7.2]{1007.3955}. Choose an ample divisor $H$, and choose $a$ and $b$ such that $L'=aL- bH$ satisfies $\Bs(|L'|)=\mathbf{B}_{+}(L)$.  

Suppose there is a point $x\in X$ which is not contained in $Y$.  Then since $Y$ is the base locus 
of $L'$, there is a section of $L'$ which does not vanish at $x$, and let $X'$ be the vanishing of this section.  Then by \ref{restrict} $L$ is not $q$-ample on $X'$.  The process only terminates when $X'=Y$, and it must terminate because $X$ was a noetherian topological space.
\end{proof}

\section{Towards a Numerical Criterion for $q$-ample Line Bundles}
The cone of ample line bundles in $N^1(X)$ has a nice description in terms of the geometry of curves in $X$ due to a theorem of Kleiman.  (See for example \cite[1.4.23]{MR2095471}.)

\begin{theorem} \label{K} (Kleiman's criterion)
Let $\text{Nef}(X)$ be the cone of nef divisors.  $\text{Nef}(X)$ is a closed cone, and the cone of ample divisors is the interior of $\text{Nef}(X)$.
\end{theorem}

One would like similar criteria to test the $q$-amplitude of $L$.  A duality argument gives a criterion for the $(n-1)$-ample cone:

\begin{theorem}\label{n-1 cone}\cite[Thm 9.1]{1007.3955}
On a variety $X$, the $(n-1)$-ample cone is the negative of the complement of the pseudoeffective cone.
\end{theorem}

The Kleiman criterion says that $L$ is in the closure of the ample cone iff $-L$ is not big on any curve.  Theorem \ref{n-1 cone} says that $L$ is in the closure of the $(n-1)$-ample cone iff $-L$ is not big on $X$, which is the only subvariety of $X$ having dimension $n$.
Thus in some sense, both criteria say that to test if a divisor is in the closure of the $q$-ample cone it suffices to show that its dual is not in the big cone of any subvarieties of dimension $q+1$.  While one would hope that such a criterion holds for all $q$, we will see in \ref{ex} an example of Totaro which shows this fails for even the case of $3$-folds.  However, if we also require the divisor to be big, we may combine Corollary \ref{big} with a modification of the duality argument to yield Corollary \ref{criterion}.

\begin{proof}[Proof of Corollary \ref{criterion}]
Certainly if $L$ is $(n-2)$-ample on $X$ it is $(n-2)$-ample on every subvariety.  For the other direction, using $\ref{big}$ if $L$ fails to be $(n-2)$-ample we have an effective Cartier divisor $D$ on which $L$ is not $(n-2)$-ample.  By \cite[Cor 7.2]{1007.3955} we may assume $D$ is reduced.  Since $X$ is nonsingular, $D$ is a still a Cartier divisor, and the dualizing sheaf $\mathcal{K}_D$ is a line bundle given by $\K_D=(\K_X\otimes\mcal{O}(D))|_{D}$.

Let $D_i$ be the components of $D$, and let $f:\coprod D_i \to D$ be the canonical map.  Then the map $\mathcal{O}_D\to f_{*}\bigoplus\mathcal{O}_{D_i}$ is injective, and so yields an injective map $H^0(D,J)\to \bigoplus H^0(D_i,J|_{D_i})$ for any line bundle $J$ on $D$.  Suppose $-L$ is not pseudoeffective on any of the $D_i$.  Then for any line bundle $J$ and sufficiently large $m$ depending on $J$, $ H^0(D_i,\mathcal{O}(J|_{D_i})=0$, so $H^0(D,\mathcal{O}(J-mL))=0$.

It follows by duality that $H^{n-1}(D, \mathcal{K}_D\otimes\mcal{O}(mL-J))=0$ for any line bundle $J$ and sufficiently large $m$.  But by \cite[Thm 7.1]{1007.3955} this means $L$ is $(n-2)$-ample on $D$, a contradiction.
\end{proof}
\section{Totaro's Example}\label{ex}
In this section we reproduce Totaro's example from \cite{1007.3955} of a line bundle $L$ on a smooth toric Fano $3$-fold $X$ such that $L$ is not in the closure of the $1$-ample cone of $X$, but $L$ is in the closure of the $1$-ample cone of every proper subvariety of $X$.  Our goal is investigate what sort of additional obstacles beyond the numerical criterion must be considered to say when an effective bundle is $q$-ample.  
\begin{definition}
A line bundle $L$ on $X$ is called $q$-nef if for every dimension $q+1$ subvariety $V\subset X$ the restriction of $-L$ to $V$ is not big.
\end{definition}

The $q$-nef cone is a closed cone in $N^1(X)$.  By Theorem \ref{n-1 cone}, a $q$-ample bundle must be $q$-nef.  Also, when $q=0$ or $q=n-1$, the $q$-ample cone is the interior of the $q$-nef cone.
Let $X$ be the projectivization of the rank $2$ vector bundle $\mathcal{O}\oplus\mathcal{O}(1,-1)$ on $\P^1\times \P^1$.  Then $X$ is a smooth toric Fano $3$-fold.  
One can show that the corresponding fan $\Sigma$ in $\Z^3\otimes \R$ has rays 
\[
f_1=(0,0,-1),f_2=(0,0,1), f_3=(1,0,1), f_4=(0,1,-1), f_5=(-1,0,0), f_6=(0,-1,0) 
\]
The two dimensional cones are given by
\[
(13),(14),(15),(16),(23),(24),(25),(26),(34),(36),(45),(46)
\]
The maximal cones are
\[
(134),(136),(145),(146),(234),(236),(245),(246)
\]
\begin{figure}
\includegraphics[width=0.35\textwidth]{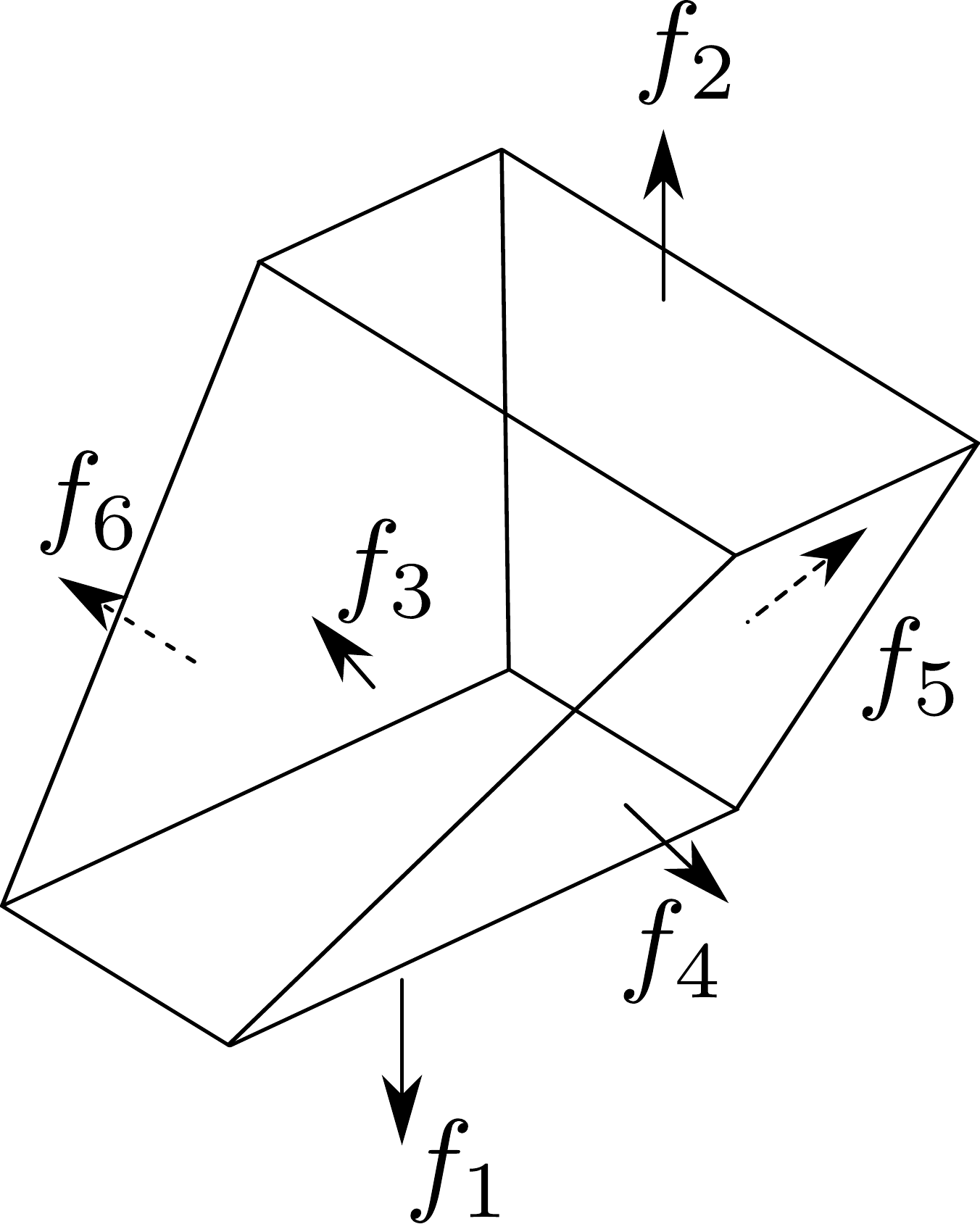}
\caption{The dual polytope to $\Sigma$}
\label{polytope}
\end{figure}Line bundles on $X$ are given by piecewise linear functions on $\Sigma$ which are integral linear functions on each cone.  Let $\langle\Sigma(1)\rangle$ be the $\R$ vector space spanned by the rays of $\Sigma$.  Since $X$ is simplicial we have an identification 
\[
\Pic\otimes\R \cong \langle\Sigma(1)\rangle^*/(\Z^3\otimes \R)^*
\]

Write $F_i$ for the function which sends $f_i$ to $1$ and $f_{j, j\neq i}$ to $0$.  Then we can identify $F_i$ with the divisor which is the closure of the torus orbit corresponding to the ray $f_i$.  Let $L=3F_1+3F_2-F_3-F_4-F_5-F_6$.  Then $L$ is not in the closure of the $1$-ample cone, but $L$ is $1$-nef.

To see that $L$ is not in the closure of the $1$-ample cone it suffices to show that a positive twist of $L$ is not $1$-ample.  For example, take $H=F_1+F_2+F_3+F_4+F_5+F_6$.
Then for any sufficiently small rational $\lambda>0$, a large integral multiple of $L+\lambda H$ has a nonvanishing $H^2$.  This follows from the formula for cohomology of line bundles given in \cite[p. 74]{MR1234037}, along with the fact that the rays with negative coefficients form a nontrivial $1$-cycle in $|\Sigma|\setminus \{0\}$.

The $1$-nef cone of a toric variety consists of divisors whose restriction to each torus invariant surface is not the negative of a big divisor.  It can be shown that $L$ is $1$-nef by restricting to each $F_i$.  As an example we explicitly work out the restriction of $L$ to $F_1$.  

The divisor $F_1$ is a toric variety and its fan is given by $\Sigma_{F_1}= \text{Star}(f_1)/\langle f_1\rangle$.  Denote the image of the ray $f_i$ in $\Sigma_{F_1}$ by $f_i'$.  This fan is isomorphic to the fan of $\P^1 \times \P^1$.  The most straightforward way of restricting $L$ to $F_1$ is to choose a linearly equivalent representative in $\langle\Sigma(1)\rangle^*$ which vanishes on $f_1$.  Take $L'=6F_2-4F_3+2F_4-F_5-F_6$.  Then the resulting piecewise linear function $\psi$ on $\Sigma_{F_1}$ has 
\[
\psi(f_3')=-4, \psi(f_4')=2, \psi(f_5')=-1, \psi(f_6')=-1
\]
This corresponds to the divisor $\mathcal{O}(1,-3)$ on $\P^1\times \P^1$, which is not the negative of a big divisor.  A similar calculation for the other $F_i$ shows that $L$ is actually $1$-nef.  

Figure \ref{effectivecone} shows a slice of $N^1(X)$, where the effective cone is shaded.  The numbers in each region are the largest $q$ such that a line bundle in the interior of that region is $q$-ample. 
\begin{figure}
\includegraphics[width=0.5\textwidth]{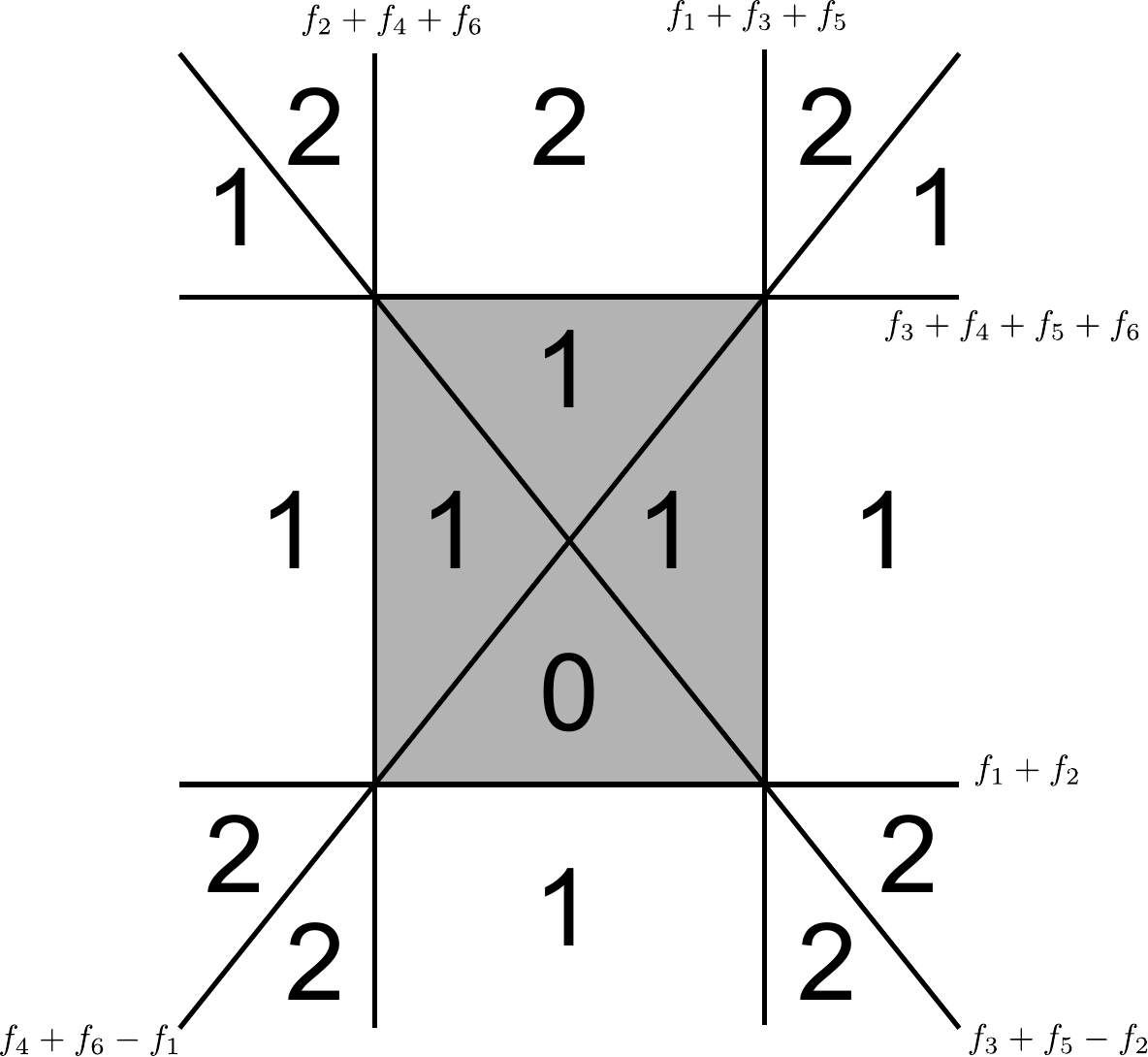}
\caption{Chambers in $N^1(X)$.  The effective cone is shaded, and each chamber is marked with the smallest $q$ such that a line bundle in the interior of the chamber is $q$-ample.  The planes are labelled by the corresponding linear dependence among rays in $\Sigma(1)$.}
\label{effectivecone}
\end{figure}

\section{Further Questions}
Let $X$ be a variety and $L$ a line bundle on $X$. When $L$ is not big, $\mathbf{B}_+(L)$ is all of $X$, and so yields no new information about whether $L$ is $q$-ample.  However, when $L$ is effective, we may hope to see other geometric consequences of $q$-amplitude reflected in the geometry of a section.
In the example in section \ref{ex}, the divisor $F_1+F_2$ is not $1$-ample, and this cannot be seen via any sort of restriction to proper subvarieties of $X$.  However, $F_1+F_2$ cannot be $1$-ample because it admits a section with disconnected zero set.
\begin{proposition}\label{disconnect}
Let $X$ be a normal irreducible Cohen-Macaulay variety of dimension $n$.  If $L$ is a line bundle on $X$ which admits a global section with disconnected zero set, then $L$ is not $(n-2)$-ample.
\end{proposition}
\begin{proof}
Let $D$ be the vanishing of section of $L$, which is disconnected.  Then we can take the infinitesimal thickening $mD$ as the vanishing of a section of $\mathcal{O}(mL)$.  Consider the restriction exact sequence:
\[
0\rightarrow\mathcal{O}(-mL)\rightarrow\mathcal{O}_X\rightarrow\mathcal{O}_{mD}\rightarrow 0
\]
Since $X$ is connected $H^0(X,\mathcal{O}_X)$ is one dimensional, but $mD$ is not connected so $H^0(mD,\mathcal{O}_{mD})$ is at least two dimensional.  Thus the associated map $H^0(X,\mathcal{O}_X)\rightarrow H^0(mD,\mathcal{O}_{mD})$ is not surjective and so taking the associated long exact sequence we see that $H^1(X,\mathcal{O}(-mL))$ is nonzero. Let $\mathcal{K}_X$ be the dualizing sheaf on $X$.  By Serre duality, $H^{n-1}(X,\mathcal{K}_X\otimes\mathcal{O}(-mL))$ is nonvanishing for all $m$ so $L$ is not $(n-2)$-ample.
\end{proof}

\begin{question}
Given a smooth variety $X$ with an effective line bundle $L$ which is $(n-2)$-nef and such that there is a neighborhood $U$ in $N^1(X)$ that no line bundle in $U$ admits a section with disconnected vanishing set, must $L$ be $(n-2)$-ample?
\end{question} 
One possible way to interpret Proposition \ref{disconnect} is as a sort of Lefschetz hyperplane theorem for $(n-2)$-ample divisors.  Bott has proved the following generalization of the Lefschetz hyperplane theorem:
\begin{theorem} \cite[Thm III]{MR0215323}
Let $X$ be a smooth variety of dimension $n$, and $L$ a line bundle which admits a Hermitian metric whose curvature form has at least $n-q$ positive eigenvalues (counted with multiplicity) at every point.  Suppose also that $Y$ is the vanishing set of a section of $L$. Then $X$ is obtained from $Y$ as a topological space by attaching cells of dimension at least $n-q$.
\end{theorem}

A line bundle is called $q$-positive if it admits such a Hermitian metric.  If $Y$ has `too much' homology in dimension $n-q-2$ it cannot be a section of a $q$-positive line bundle. It is a well known theorem of Andreotti and Grauert \cite{MR0150342} that a $q$-positive line bundle is $q$-ample.  The problem of determining when the converse holds was posed by \cite{MR1360502}, but little progress had been made until recently.  Ottem \cite{1105.2500} has given examples of line bundles which are $q$-ample but not $q$-positive when $\frac{1}{2}\text{dim}X-1<q<\text{dim}X-2$.  These examples are effective, and the analogue of the Lefschetz hyperplane theorem holds over $\mathbb{Q}$ but not $\mathbb{Z}$.  S. Matsumura  has shown in \cite{1104.5313} that if $X$ is a compact $n$ dimensional complex manifold with a K{\"a}hler form $\omega$, and $L$ is a line bundle such that the intersection $\omega^{n-1}\cdot L>0$, then $L$ is $1$-positive.
\bibliographystyle{amsalpha}      
\bibliography{bigqample}
\end{document}